\documentclass[10pt,reqno]{amsart}
\usepackage{epic}
\usepackage{arydshln}
\usepackage{color}
 \renewcommand{\a}{\alpha}
\renewcommand{\b}{\beta}

\renewcommand{\d}{{\delta}}
\newcommand{\g}{\gamma}
\newcommand{\G}{\Gamma}
\renewcommand{\l}{\lambda}

\renewcommand{\t}{\tau}

\renewcommand{\(}{\left\(}
\renewcommand{\)}{\right\)}
\renewcommand{\[}{\left\[}
\renewcommand{\]}{\right\]}

 \theoremstyle{plain}
\newtheorem{theorem}{Theorem}[section]
\newtheorem{lemma}[theorem]{Lemma}
\newtheorem{corollary}[theorem]{Corollary}

\newtheorem{example}[]{Example}
\newtheorem{remark}[]{Remark}

\begin{document}

\title[Overpartitions and singular overpartitions]
{Overpartitions and singular overpartitions}

\author{Seunghyun Seo}
\address{Department of Mathematics Education, Kangwon National University,
Chuncheon, Kangwon-do 24341, Republic of Korea} \email{shyunseo@kangwon.ac.kr}

\author{Ae Ja Yee}
\address{Department of Mathematics, The Pennsylvania State University,
University Park, PA 16802, USA} \email{auy2@psu.edu}

\dedicatory{Dedicated to Krishnaswami Alladi for his 60th birthday}

\maketitle

%\footnotetext[1]{E-mail address: auy2@psu.edu/ Phone number:1-814-865-3873.}
\footnotetext[1]{The first author was partially supported by a research grant of Kangwon National University in 2015.} \vspace{0.5in}
\footnotetext[2]{The second author was partially supported by a grant ($\#$280903) from the Simons Foundation.} \vspace{0.5in}
\footnotetext[3]{2010 AMS Classification Numbers: Primary, 05A17; Secondary, 11P81.}

\noindent{\footnotesize{\bf Abstract.}  
Singular overpartitions, which were defined by George Andrews, are overpartitions whose Frobenius symbols have at most one overlined entry in each row.  In his paper, Andrews obtained interesting combinatorial results on singular overpartitions, one of which relates a certain type of  singular overpartitions  with a subclass of overpartitions.   In this paper, we provide a combinatorial proof of Andrews's result, which answers one of his open questions.}


\section{Introduction}

 A Frobenius symbol for $n$ is a two-rowed array \cite{gea109, yee12}:
 \begin{equation*}
\left( \begin{matrix} a_1 & a_2 & \cdots & a_\d\\ 
b_1 & b_2 & \cdots & b_\d\end{matrix} \right)
\end{equation*}
where $\sum_{t=1}^{\d}(a_t +b_t+1)=n$,  $a_1>a_2>\cdots>a_\d\ge0$, and
$b_1>b_2>\cdots>b_\d\ge0$.  There is a natural one-to-one correspondence between the Frobenius symbols for $n$ and the  ordinary partitions of $n$ (see Section~\ref{f-symbol}).  Thus  a Frobenius symbol for $n$ is another representation of an ordinary partition of $n$. 

An overpartition of $n$ is a partition in which the first occurrence of a part may be overlined \cite{corteel}.  For an overparition, one can define  the corresponding Frobenius symbol by allowing overlined entries in a similar way.  It should be noted that the Frobenius symbol of an overpartition is defined in a different way in \cite{corteel, lovejoy}.  

Recently, George Andrews introduced a certain subclass of overpartitions, namely singular overpartitions which are Frobenius symbols with  at most one overlined entry in each row  \cite{gea303}.  
For integers $k,i$ with $k\ge 3$ and  $1\le i < k$,  Andrews found interesting combinatorial and arithmetic properties of $(k,i)$-singular overpartitions, which are singular overpartitions  with some restrictions subject to $k$ and $i$. Because of the complexity of the restrictions, we defer the exact definition to Section~\ref{singular}.  One of the main results of Andrews in \cite{gea303} is  the following. 
\begin{theorem}[Andrews, \cite{gea303}]\label{thm1}
The number of $(k,i)$-singular overpartitions of $n$ equals the number of overpartitions of $n$ in which no part is divisible by $k$ and only parts congruent to $\pm i$ mod $k$ may be overlined.
\end{theorem}
Equivalently, 
\begin{equation*}
\sum_{n=0}^{\infty}\overline{Q}_{k,i}(n) q^n=  \prod_{n=0}^{\infty} \frac{(1+q^{nk+i})(1+q^{(n+1)k-i})}{(1-q^{nk+1})(1-q^{nk+2})\cdots (1-q^{nk+k-1})}, 
\end{equation*}
where $\overline{Q}_{k,i}(n)$ is the number of $(k,i)$-singular overpartitions of $n$. 

Andrews concluded his paper with four open questions.  The first question is to prove Theorem~\ref{thm1} bijectively. The primary purpose of this paper is to provide an answer to the first question.  His second question is indeed a special case of the first one.  Consequently, the second question will be settled as well.  In additon, we obtain a refined version of Theorem~\ref{thm1}, namely Theorem~\ref{thm4.1} in Section~\ref{sec4}. 

This paper is organized as follows. In Section~\ref{sec2}, necessary definitions and maps are reviewed. In Section~\ref{sec3}, an enumeration formula for subclasses of $(k,i)$-singular overpartitions is given (see Theorem~\ref{thm3.1}), and finally, a combinatorial proof of Theorem~\ref{thm1} will be presented in Section~\ref{sec4}. 

 \iffalse
 Let $\overline{C}_{k.i}(n)$ be the number of overpartitions of $n$ in which no part is divisible by $k$ and only parts congruent to $\pm i$ mod $k$ may be overlined. Then 
\begin{theorem}
Let $\overline{Q}_{k,i}(n,m)$ be the number of $(k,i)$-singular overpartitions
\begin{equation*}
\overline{Q}_{k,i}(n)=\overline{C}_{k,i}(n),
\end{equation*}
 is the number singular overpartitions of $n$ and $\overline{C}_{k,i}(n)$ is the number of overpartitions of $n$ in which no part is divisible by $k$ and only parts congruent to $\pm i$ mod $k$ may be overlined.
\end{theorem}
\fi

\section{Preliminaries}\label{sec2}

In this section, we provide some definitions and bijections that are needed in later sections.  

\subsection{Definitions} 
For a partition or overpartition $\l$, we write it as $\l=(\l_1, \l_2, \ldots)$ with $\l_1\ge \l_2\ge \cdots$.   We denote by $|\l|$ the sum of parts, and by $\ell(\l)$ the number of parts. 

The conjugate $\l'$ of a partition $\lambda$ is the partition resulting from the reflection of the Ferrers graph of $\l$ about the main diagonal. 

For a positive integer $k$, we also define $k\lambda$ as the partition whose parts are $k$ times  each part of $\lambda$.   For instance, let $\lambda=(3, 3, 2, 1)$. Then $5\lambda=(15, 15, 10, 5)$. 

Let $\lambda$ and $\mu$ be two partitions. Then we define the union $\lambda\cup \mu$ as the partition consisting of all the parts of $\lambda$ and $\mu$.  

We denote the number of partitions of $n$ by $p(n)$.  We also denote the partition with no parts by $\emptyset$.  For further standard definitions, see \cite{gea1}.

\subsection{Frobenius symbol}\label{f-symbol} Recall the definition of a Frobenius symbol for $n$ in Introduction.  For a partition $\lambda$ of $n$, let $\delta$ be the largest integer such that $\lambda_{\delta}-\delta\ge 0$, i.e., $\delta$ is the side of the Durfee square of $\lambda$. We now consider the following two-rowed array:
\begin{equation*}
\left( \begin{matrix} \l_1-1& \l_2-2 & \cdots & \l_{\delta} -\delta \\ \l'_1-1& \l'_2-2 & \cdots & \l'_{\delta} -\delta \end{matrix}\right).
\end{equation*}
Clearly, this satisfies the conditions for Frobenius symbols for $n$, and this is reversible. Thus there is a unique Frobenius symbol associated with $\lambda$.   For instance,
the associated Frobenius symbol of the partition $(7, 5, 5, 3, 2, 2, 1)$ is
\begin{equation*}
\left(\begin{matrix} 6 & 3& 2\\ 6 & 4 & 1\end{matrix} \right). 
\end{equation*}
%For the empty partition,  $\left(\right)$ denotes its Frobenius symbol.

\subsection{$(k,i)$-parity blocks and anchors}\label{block} 
Throughout this paper, we assume that $k$ and $i$ are integers such that  $k\ge 3$ and $1 \le i \le k-1$. %, and we set $j:=k-i$.

For a partition $\l$, by abuse of the notation, we will denote its Frobenius symbol by $\l$. A column $\begin{matrix} a_t\\ b_t \end{matrix}$ of $\l$ is called $(k,i)$-positive if $a_t-b_t \ge k-i-1$ and called {$(k,i)$-negative} if $a_t-b_t \le -i+1$. If $-i+2 \le a_t-b_t \le k-i-2$, we call the column $(k,i)$-neutral. 

If two columns are both $(k,i)$-positive or both $(k,i)$-negative, we shall say that they have the same parity. 

We now divide $\l$ into $(k,i)$-{parity blocks}. 
These  are sets of contiguous columns maximally extended to the right, where all the entries have either the same parity or neutral.  

We shall say that a parity block is neutral if all columns are neutral. Owing to the maximality condition this can only occur if all the columns of $\l$ are neutral. In all other cases, we shall say that a block is positive (or negative) if it contains no negative (or positive, resp.) columns.
\iffalse
We shall say that a $(k,i)$-parity block is positive ($P$) if its anchor is positive
if it begins with a positive column, negative ($N$)  if it consists of neutral columns only. If a parity block begins with the $\ell$-th anchor, then we denote the block by $B_{\ell}$. If there exists a neutral block, it should begin with the first column. We denote the neutral block by $B_{0}$.  For convention, we allow $B_{0}$ can be empty, so that every Frobenius symbol begins with neutral block $B_0$.

We now define the $(k,i)$-{\it anchors} of $\l$ as follows.  The first anchor is the first non-neutral column. The second anchor is the first non-neutral column such that it is on the right of the first anchor  and its parity is opposite to that of the first anchor. The third anchor is the first non-neutral column such that it is on the right of the second anchor and its parity is opposite to that of the second anchor. We successively choose anchors in  this manner until we run out of the columns.
\fi
For instance, consider the following Frobenius symbol: 
\begin{equation*} %\label{ex:main}
%\setcounter{MaxMatrixCols}{20}
%\l=\left( \begin{matrix}
%31 & 28 & 27 & 25 & 22 & 18 & 16 & 14 & 13 & 9  & 8 & 7 & 6 & 4 & 1 & 0 
%\\ 
%30 & 28 & 25 & 24 & 20 & 19 & 16 & 15 & 12 & 11 & 8 & 7 & 4 & 3 & 2 & 0 
%\end{matrix} \right).
\left( \begin{array}{cccccccccccccccc}
31 & 28 & 27 & 25 & 22 & 18 & 16 & 14 & 13 & 9  & 8 & 7 & 6 & 4 & 1 & 0 
\\ 
30 & 28 & 25 & 24 & 20 & 19 & 16 & 15 & 12 & 11 & 8 & 7 & 4 & 3 & 2 & 0
\end{array} \right). 
\end{equation*}
The $(3,1)$-parity blocks are
\begin{equation*}
\left( \begin{array}{c|c|ccc|ccc|c|ccc|cc|cc}
31 & 28 & 27 & 25 & 22 & 18 & 16 & 14 & 13 & 9  & 8 & 7 & 6 & 4 & 1 & 0 
\\ 
30 & 28 & 25 & 24 & 20 & 19 & 16 & 15 & 12 & 11 & 8 & 7 & 4 & 3 & 2 & 0 
\end{array} \right),
\end{equation*}
%where the parity type of blocks is 
%$PNPNPNPN.$
and 
 the $(5,2)$-parity blocks are
\begin{equation} \label{example1}
\left( \begin{array}{ccccc|ccccccc|cc|cc}
31 & 28 & 27 & 25 & 22 & 18 & 16 & 14 & 13 & 9  & 8 & 7 & 6 & 4 & 1 & 0 
\\ 
30 & 28 & 25 & 24 & 20 & 19 & 16 & 15 & 12 & 11 & 8 & 7 & 4 & 3 & 2 & 0 
\end{array} \right).
\end{equation}

For a non-neutral block, we now define its {anchor} as the first non-neutral column.

\subsection{$(k,i)$-singular overpartitions}\label{singular}

We are ready to define $(k,i)$-singular overpartitons. A Frobenius symbol is $(k,i)$-singular if it satisfies one of the following conditions:
\begin{itemize}
\item there are no overlined entries; 
\item if there is one overlined entry on the top row, then it occurs in the anchor of a positive block;
\item if there is one overlined entry on the bottom row, then it occurs in the anchor of a negative block;
\item if there are two overlined entries, then they occur in adjacent anchors with one on the top row of the positive block and the other on the bottom row of the negative block. 
\end{itemize} 
For the Frobenius symbol  in \eqref{example1}, the following are all the $(5,2)$-singular with exactly two overlined entries:
\begin{align*}
\left( \begin{array}{ccccc|ccccccc|cc|cc}
31 & 28 & \overline{27} & 25 & 22 & 18 & 16 & 14 & 13 & 9  & 8 & 7 & 6 & 4 & 1 & 0 
\\ 
30 & 28 & 25 & 24 & 20 & \overline{19} & 16 & 15 & 12 & 11 & 8 & 7 & 4 & 3 & 2 & 0 
\end{array} \right),\\
\left( \begin{array}{ccccc|ccccccc|cc|cc}
31 & 28 & {27} & 25 & 22 & 18 & 16 & 14 & 13 & 9  & 8 & 7 & \overline{6} & 4 & 1 & 0 
\\ 
30 & 28 & 25 & 24 & 20 & \overline{19} & 16 & 15 & 12 & 11 & 8 & 7 & 4 & 3 & 2 & 0 
\end{array} \right),\\
\left( \begin{array}{ccccc|ccccccc|cc|cc}
31 & 28 & {27} & 25 & 22 & 18 & 16 & 14 & 13 & 9  & 8 & 7 & \overline{6} & 4 & 1 & 0 
\\ 
30 & 28 & 25 & 24 & 20 &19 & 16 & 15 & 12 & 11 & 8 & 7 & 4 & 3 &  \overline{2} & 0 
\end{array} \right).
\end{align*}

\subsection{Dyson map} \label{dyson}
For a partition $\l$, the rank of $\l$ is the largest part minus the number of parts, i.e.,
$$rank(\l):= \l_1-\ell(\l).$$
We remark that the rank of $\l$ is equal to $a_1 -b_1$ if $\begin{matrix} a_1\\b_1\end{matrix}$ is the first column of the Frobenius symbol of $\lambda$. For convenience, we define $rank(\emptyset)=0$. 

Freeman Dyson defined a map to prove a symmetry in partitions \cite{dyson}.  Here, we use the description of the Dyson map in \cite{boulet, pak}.
For a partition $\l$ of $n$ with $rank(\l)\le r$, we subtract $1$ from each part of $\lambda$ and then add a part of size $r-1+ \ell(\l)$.  We call this the Dyson map and denote it by $d_r$.  
\begin{remark}\label{remark1}
It can be easily checked that $d_r(\l)$ is a partition of $n+r-1$ with  $rank(d_r(\l))  \ge r-2$.
\end{remark}
%correspond it to a partition $\mu$ of $n+r-1$ with $r(\mu)\ge r-2$, by removing the first column of $\l$ and adding a part of size $r-1+\mbox{(the number of parts of $\l$)}$. This bijection is called the Dyson map and denoted by $d_r$.
%See Figure for an example. (Ferrer diagram).

We now describe the Dyson map with Frobenius symbol for later use.  Consider 
\begin{equation*}
\l=\left( \begin{matrix} a_1 & a_2 & \cdots & a_\d\\ b_1 & b_2 & \cdots & b_\d\end{matrix} \right),
\end{equation*}
with $rank(\l) \le r$.
When $\d\ge2$, $d_r(\l)$ is given by
\begin{itemize}
\item $\left( \begin{matrix} 
b_1+r-1 & a_1-2 & \cdots & a_{\d-2}-2 & a_{\d-1}-2 & a_{\d}-2\\ 
b_2+2 & b_3+2 & \cdots & b_\d+2 & 1 & 0
\end{matrix} \right)$\quad if $a_\d \ge 2$,
\item $\left( \begin{matrix} 
b_1+r-1 & a_1-2 & \cdots & a_{\d-2}-2 & a_{\d-1}-2 \\ 
b_2+2 & b_3+2 & \cdots & b_\d+2 & 1 
\end{matrix} \right)$\quad if $a_\d =1$,
\item $\left( \begin{matrix} 
b_1+r-1 & a_1-2 & \cdots & a_{\d-2}-2 \\ 
b_2+2 & b_3+2 & \cdots & b_\d+2 
\end{matrix} \right)$\quad if $a_\d=0$, $a_{\d-1}=1$,
\item $\left( \begin{matrix} 
b_1+r-1 & a_1-2 & \cdots & a_{\d-2}-2 & a_{\d-1}-2 \\ 
b_2+2 & b_3+2 & \cdots & b_\d+2 & 0
\end{matrix} \right)$\quad if $a_\d=0$, $a_{\d-1}\ge 2$.
\end{itemize}
When $\d=1$, $d_r(\l)$ is given by 
\begin{itemize}
\item $\left( \begin{matrix} 
b_1+r-1 & a_1-2\\ 
1 & 0
\end{matrix} \right)$\quad if $a_1 \ge 2$,
\item $\left( \begin{matrix} 
b_1+r-1 \\ 
1
\end{matrix} \right)$\quad if $a_1=1$,
\item $\left( \begin{matrix} 
b_1+r-1 \\ 
0
\end{matrix} \right)$\quad if $a_1=0$ and $b_1 \ge 1-r$,

\item 
%$\left( \begin{matrix} \\ 
%\end{matrix} \right)$ 
\quad $\emptyset$ \quad if $a_1=0$ and $b_1 = -r$.
\end{itemize}
Note that the upper left entry of $d_r(\l)$ is $b_1 +r-1$ if $d_r(\l)$ is nonempty.

\subsection{Shift and Shifted Conjugate }
Given an integer $u$, a shift map $s_u$ is defined as follows:
\begin{equation*}
\left( \begin{matrix} a_1 & a_2 & \cdots & a_\d\\ b_1 & b_2 & \cdots & b_\d\end{matrix} \right)
\stackrel{s_u}{\longrightarrow}
\left( \begin{matrix} 
a_1-u & a_2-u & \cdots & a_\d-u\\ 
b_1+u & b_2+u & \cdots & b_\d+u
\end{matrix} \right).
\end{equation*}
Similarly, a shifted conjugate map $c_u$ is defined as follows:
\begin{equation*}
\left( \begin{matrix} a_1 & a_2 & \cdots & a_\d\\ b_1 & b_2 & \cdots & b_\d\end{matrix} \right)
\stackrel{c_u}{\longrightarrow}
\left( \begin{matrix} 
b_1-u & b_2-u & \cdots & b_\d-u\\ 
a_1+u & a_2+u & \cdots & a_\d+u
\end{matrix} \right).
\end{equation*}

\begin{remark} \label{remark2}
It follows from the definitions that $s_{u}^{-1}=s_{-u}$ and $c_{u}^{-1}=c_{u}$.  Also, for the Frobenius symbol of the empty partition $\emptyset$,  we define $s_u( \emptyset )=c_u(\emptyset)=\emptyset$. 
\end{remark}

\subsection{Wright map}\label{wright}
%Let us replace $q$ by $q^k$ and $z$ by $z^{-1}q^{-i}$ in the Jacobi triple product identity. 
\iffalse
Note that 
\begin{align}
(-q^{i};q^k)_{\infty} (-q^{j};q^k)_{\infty}=\sum_{(\mu^1, \mu^2)}  %(-z)^{\ell(\mu^1)-\ell(\mu^2)}  
q^{|\mu^1|+|\mu^2|}, \label{2.21}
\end{align}
where the sum is over all pairs of partitions $(\mu^1, \mu^2)$ into distinct parts, and the parts of $\mu^1$ and $\mu^2$ are congruent to $i$ and $j$ mod $k$,  respectively. We now consider pairs $(\mu^1, \mu^2)$ with $l(\mu^1)-l(\mu^2) =\delta$ only. Then the generating function is
\begin{align}
\sum_{(\mu^1, \mu^2) \atop l(\mu^1)-l(\mu^2) =\delta } q^{|\mu^1|+|\mu^2|} =\frac{q^{k(\delta^2-\delta)/2+\delta i}}{(q^k;q^k)_{\infty}}. \label{2.22}
\end{align}
This can be shown as follows. 
\fi
The Wright map is a bijection between pairs of partitions into distinct parts and pairs of an ordinary partition and a triangular number (see \cite{pak, wright, yee41}).  However, in this paper, we give a modified version of the map  using Frobenius symbol for our purpose.  We denote the map by $\varphi$. 

Let $\mu^1$ be a partition into distinct parts congruent to $i$ mod $k$  and $\mu^2$ be a partition into distinct parts congruent to $(k-i)$ mod $k$, namely
\begin{align*}
\mu^1&=(ka_1+i, ka_2+i, \ldots, ka_{s+m}+i), \\
\mu^2&=(kb_1+(k-i), kb_2+(k-i),\ldots, kb_s+(k-i)),
\end{align*}
where $a_1>a_2>\cdots>a_{m+s}\ge 0$ and $b_1>b_2>\cdots>b_m\ge 0$. %$m$ is the number of parts of $\mu^1$ minus the number of parts of $\mu^2$.  

Suppose that $m\ge 0$.  We consider the following Frobenius symbol
\begin{align*}
\mu=\left( \begin{matrix}%{cccccccccccccccc}
a_{1+m} & a_{2+m} & \ldots & a_{s+m} \\ b_1 & b_2 & \ldots & b_s
\end{matrix}\right).
\end{align*}
We now take $\nu=(a_1-m+1, a_2-m+2, \ldots, a_m)$. Then since $a_1>a_2>\cdots$, it is clear that $\nu$ is a partition.   We define $\varphi(\mu^1, \mu^2)=(k(\nu\cup \mu),m) $.

For example, let $k=5$ and $i=2$. If $\mu^1=(37, 27, 22, 7)$ and $\mu^2=(18, 13)$, then $m=2$, and we obtain  
\begin{equation*}
\mu=\left(\begin{matrix}  4 & 1 \\ 3 & 2 \end{matrix} \right)=(5, 3, 2, 2),\quad \nu=(6, 5). 
\end{equation*}
%$\nu=(6, 5)$. 
Thus
\begin{equation*}
\varphi(\mu^1, \mu^2)=(5(6, 5, 5, 3, 2, 2), 2)=((30, 25, 25, 15, 10, 10), 2).
\end{equation*}

Similarly, if $m<0$, 
we consider the following Frobenius symbol
\begin{align*}
\mu=\left( \begin{matrix}%{cccccccccccccccc}
b_{1-m} & b_{2-m} & \ldots & b_s\\
a_{1} & a_{2} & \ldots & a_{s+m} 
\end{matrix}\right).
\end{align*}
We now take $\nu=(b_1+m+1, b_2+m+2, \ldots, b_{-m})$. Then since $b_1>b_2>\cdots$, it is clear that $\nu$ is a partition.   We define $\varphi(\mu^1, \mu^2)=(k(\nu\cup \mu)',m)$.

We can easily check that $|\mu^1|+|\mu^2|=k(|\nu|+|\mu|)+k\binom{m}{2}+i m$, and we omit the details.

\begin{remark} \label{remark3}
 We note that the Wright map proves that the number of such pairs of partitions $\mu^1, \mu^2$ with $|\mu^1|+|\mu^2|=n$ and $\ell(\mu^1)-\ell(\mu^2)=m$ is 
\begin{equation*}
p\bigg(\frac{n-k \binom{m}{2} - i m}{k} \bigg).
\end{equation*}
\end{remark}

\section{Singular overpartitions and dotted blocks}\label{sec3}

\subsection{Dotted parity blocks} We now introduce another representation of a $(k,i)$-singular overpartition.  We will use this representation throughout this paper. 

Let $\l$ be a $(k,i)$-singular overpartition.  We first separate all the columns before the first anchor to form a block. By the definition of parity blocks, we see that these columns must be all neutral if exist. We denote this block  by $E$.  For the blocks of the remaining columns, we denote each of them by $P$ and $N$ if its anchor is positive and negative, respectively.   

If there is exactly one overlined entry in $\l$, we put a dot on the top of each of the blocks between  the first non-neutral block and the block of the overlined entry.   If there are two overlined entries in $\l$, then we put a dot on the top of each block between the second non-neutral block and the block of the last overlined entry.   

It is clear that a Frobenius symbol $\l$ is $(k,i)$-singular if 
\begin{enumerate}
\item[S1.] there are no dotted blocks, or
\item[S2.] there are consecutive dotted  blocks starting from the first non-neutral block, or
\item[S3.] there are consecutive dotted  blocks starting from the second non-neutral block.
\end{enumerate}

For instance, if a sequence of  parity blocks is $EPNPN$, then the following are all singular:
\begin{enumerate}
\item[] $EPNPN$,
\item[] $E\dot{P}NPN$, $E\dot{P}\dot{N}PN$, $E\dot{P}\dot{N}\dot{P}N$, $E\dot{P}\dot{N}\dot{P}\dot{N}$, 
\item[] $E{P}\dot{N}PN$, $E{P}\dot{N}\dot{P}N$, $E{P}\dot{N}\dot{P}\dot{N}$.
\end{enumerate}

Since there is a one-to-one correspondence between  $(k,i)$-singular overpartitions and Frobenius symbols with a sequence of parity blocks satisfying S1, S2, or S3,  we will use the latter form from now on.  

\subsection{$(k,i)$-singular overpartitions with $m$ dotted blocks}

The following theorem is one of our main results.
\begin{theorem} \label{thm3.1}
Let $m$ be a positive integer.  
\begin{enumerate}
\item The number of $(k,i)$-singular overpartitions of $n$ with exactly $m$ dotted blocks and the last dotted block negative is 
\begin{equation*}
%p\left(n-i\binom{m+1}{2}-j\binom{m}{2} \right)=
p\left(n-k\binom{m}{2}-im \right).
\end{equation*}
%if the last dotted block is negative, and
\item The number of $(k,i)$-singular overpartitions of $n$ with exactly $m$ dotted blocks and the last dotted block positive is 
\begin{equation*}
%p\left(n-j\binom{m+1}{2}-i\binom{m}{2} \right)=
p\left(n-k\binom{m+1}{2}+i m \right).
\end{equation*}
%if the last dotted block is positive.
\end{enumerate}
\end{theorem}

Note that singular overpartitions with no dotted blocks are just ordinary partitions, which with Theorem~\ref{thm3.1} yields the following corollary. 

\begin{corollary} %\label{euler}
The number of  $(k,i)$-singular overpartitions of $n$  is 
\begin{equation*}
%\sum_{m=-\infty}^{\infty} p\left(n-\binom{m+1}{2}i-\binom{m}{2}j \right)=
\sum_{m=-\infty}^{\infty} p\left(n-k\binom{m}{2}-im \right).
\end{equation*}
\end{corollary}

\proof
For $m<0$,  %since $j=k-i$, 
the number of $(k,i)$-singular overpartitions of $n$ with exactly $|m|$ dotted blocks and the last block positive is
\begin{equation*}
p\left(n-k\binom{m}{2}-im \right),
\end{equation*}
which completes the proof.
\endproof

To prove Theorem~\ref{thm3.1},  we will construct a bijection in Section~\ref{psi}. However, we first need the following two lemmas.  

\begin{lemma} \label{lem:cd} Given integers $f, g, h$ with $g\ge 1$, $2g\ge f+1$, $h\ge f$,
consider two Frobenius symbols
\begin{equation*} 
L=\left( \begin{matrix} a_1 &  \cdots & a_{t-1} \\ 
b_1 & \cdots &b_{t-1} \end{matrix} \right)\quad\mbox{and}\quad
R=\left( \begin{matrix} \a_1 &  \cdots & \a_{\t}  \\ 
\b_1 & \cdots & \b_{\t}\end{matrix} \right)\neq \emptyset,
\end{equation*}
such that \begin{itemize}
\item[i)] $a_y - b_y \ge f$ for all $1\le y \le t-1$,
\item[ii)] $\a_1 - \b_1 \le f-2g+1$,
\item[iii)] $a_{t-1}> \a_1+g-1$,
\item[iv)] $b_{t-1}> \b_1-g+1 \ge 0$.
\end{itemize}
Let
\begin{equation*} 
\mu=
\left( \begin{matrix} \mu_{11} &  \mu_{12} & \cdots & \mu_{1{\d}'}\\ 
\mu_{21} &  \mu_{22} & \cdots & \mu_{2{\d}'}\end{matrix} \right)
:=
c_{g-f+1}(L)\,d_{f-2g+1}(R),
\end{equation*}
where the first $t-1$ columns of $\mu$ are from $c_{g-f+1}(L)$ and the rest are from $d_{f-2g+1}(R)$.  
Then the following are true.
\begin{enumerate}
\item $\mu$ is a Frobenius symbol. \label{lem1:wd}
\item  $\mu_{1y} - \mu_{2y} \le f-2g-2$ for all $1\le y \le t-1$ and $\mu_{1t} - \mu_{2t} \ge f-2g-1$ if $\mu_{1t}$ and $\mu_{2t}$ exist. \label{lem1:rk}
\item %If $a_1-b_1 \ge \tilde{f}$ then 
$rank(\mu)\le -h+2f -2g-2$ if  $L\neq \emptyset$ and $rank(L)\ge h$. \label{lem1:rank}
\item The correspondence from $(L,R)$ to $\mu$ is reversible. \label{lem1:re}
\item $|L|+|R|-|\mu|=2g-f$. \label{lem1:wt}
%\item $b_{t-1}\ge g-f+1$. \label{lem1:it}
\end{enumerate}
\end{lemma}
\begin{proof}
First note that since $rank(R)=\a_1-\b_1\le f-2g+1$, $d_{f-2g+1}(R)$ is well defined.  We now prove each of the five statements. 

 \eqref{lem1:wd}
If $L=\emptyset$, then $\mu=d_{f-2g+1}(R)$ is obviously
a Frobenius symbol because the Dyson map is defined on partitions.
Now assume $L\ne\emptyset$.
Then the last column of  $c_{g-f+1}(L)$ is  
$\begin{matrix} b_{t-1}-g+f-1\\ 
a_{t-1}+g-f+1 \end{matrix}$.  
\begin{itemize}
\item
If $d_{f-2g+1}(R)\ne \emptyset$, then its first column is
$\begin{matrix} \b_1+f-2g\\ \g \end{matrix}$, 
where $\g$ is $0$, $1$, or $\b_{2}+2$. (See Section~\ref{dyson}).
Since $b_{t-1}>\beta_1-g+1$, we have 
\begin{align*}
&b_{t-1}-g+f-1 > \b_1+f-2g,
\end{align*}
from which with $a_{t-1}-b_{t-1}\ge f$,  it follows that
\begin{equation*}
a_{t-1}+g-f+1 \ge b_{t-1}+g+1 > \b_1+2 > \g.
\end{equation*}
Thus $\mu=c_{g-f+1}(L)\,d_{f-2g+1}(R)$ is a Frobenius symbol.

\item
If $d_{f-2g+1}(R)=\emptyset$,  
then $R=\left(\begin{matrix} 0\\ 2g-f-1 \end{matrix}\right)=
\left(\begin{matrix} \a_1\\ \b_1 \end{matrix}\right)$. Since $b_{t-1}>\beta_1-g+1$,
\begin{equation*}
b_{t-1}-g+f-1 \ge \b_1 - 2g +f +1 = 0. 
\end{equation*}
Also, since $a_{t-1}>\alpha_1+g-1$ and $2g\ge f+1$, 
\begin{equation*}
a_{t-1}+g-f+1 \ge \a_{1}+2g-f +1 \ge 2. 
\end{equation*}
Thus $\mu=c_{g-f+1}(L)$ is a Frobenius symbol.
\end{itemize}

(2)
For $1\le y \le t-1$, since $a_y-b_y\ge f$, 
\begin{equation*}
\mu_{1y}-\mu_{2y}=(b_y-a_y)-2g+2f-2\le -f-2g+2f-2.
\end{equation*}

Also, by Remark~\ref{remark1}, 
\begin{equation*}
\mu_{1t}-\mu_{2t}=rank(d_{f-2g+1}(R))\ge (f-2g+1)-2.
\end{equation*}

(3) If $rank(L)\ge h$ then $a_1-b_1 \ge h$. Thus
$$
rank(\mu)=\mu_{11}-\mu_{21}=(b_1-a_1)-2g+2f-2\le -h-2g+2f-2.
$$

(4) Consider a Frobenius symbol $\mu=\left( \begin{matrix} \mu_{11} &  \mu_{12} & \cdots & \mu_{1{\d}'}\\ 
\mu_{21} &  \mu_{22} & \cdots & \mu_{2{\d}'}\end{matrix} \right)$
satisfying \eqref{lem1:rk}. 
Set 
$$
L'=\left( \begin{matrix} \mu_{11} & \cdots & \mu_{1(t-1)}\\ 
\mu_{21} & \cdots & \mu_{2(t-1)}\end{matrix} \right) 
\quad\mbox{and}\quad
R'=\left( \begin{matrix} \mu_{1t}  & \cdots & \mu_{2{\d}'}\\ 
\mu_{2t} & \cdots & \mu_{2{\d}'}\end{matrix} \right). 
$$

Then $L=c_{g-f+1}(L')$ and $R=d_{f-2g+1}^{-1}(R')$ are desired Frobenius symbols, 
so \eqref{lem1:re} holds. 

(5) Finally, since $|c_{g-f+1}(L)|=|L|$ and $|d_{f-2g+1}(R)|=|R|-(2g-f)$, \eqref{lem1:wt}  holds true.
%Finally, \eqref{lem1:it} is derived in the proof of eqref{lem1:wd}.
\end{proof}

\begin{lemma} \label{lem:sd} Given integers $f, g, h$ with $g\ge 1$, $2g\ge f+1$, $h\le f$,
consider two Frobenius symbols
\begin{equation*} 
L=\left( \begin{matrix} a_1 &  \cdots & a_{t-1} \\ 
b_1 & \cdots &b_{t-1} \end{matrix} \right)\quad\mbox{and}\quad
R=\left( \begin{matrix} \a_1 &  \cdots & \a_{\t}  \\ 
\b_1 & \cdots & \b_{\t}\end{matrix} \right) \neq \emptyset,
\end{equation*}
such that \begin{itemize}
\item[i)] $a_y - b_y \le f$ for all $1\le y \le t-1$,
\item[ii)] $\a_1 - \b_1 \le f-2g+1$,
\item[iii)] $a_{t-1}> \b_1-g+f+1$,
\item[iv)] $b_{t-1}> \a_1+g-f-1 \ge 0$.
\end{itemize}
Let
\begin{equation*} 
\mu=
\left( \begin{matrix} \mu_{11} &  \mu_{12} & \cdots & \mu_{1{\d}'}\\ 
\mu_{21} &  \mu_{22} & \cdots & \mu_{2{\d}'}\end{matrix} \right)
:=
s_{g+1}(L)\,d_{f-2g+1}(R),
\end{equation*}
where the first $t-1$ columns of $\mu$ are from $s_{g+1}(L)$ and the rest are from $d_{f-2g+1}(R)$.  
Then the following are true.
\begin{enumerate}
\item $\mu$ is a Frobenius symbol. \label{lem2:wd}
\item  $\mu_{1y} - \mu_{2y} \le f-2g-2$ for all $1\le y \le t-1$ and $\mu_{1t} - \mu_{2t} \ge f-2g-1$ if $\mu_{1t}$ and $\mu_{2t}$ exist. \label{lem2:rk}
\item $rank(\mu)\le h-2g-2$ if $L\ne \emptyset$  and $rank(L) \le h$. \label{lem2:rank}
\item The correspondence from $(L,R)$ to $\mu$ is reversible. \label{lem2:re}
\item $|L|+|R|-|\mu|=2g-f$. \label{lem2:wt}
%\item $b_{t-1}\ge g-f+1$. \label{lem1:it}
\end{enumerate}
\end{lemma}
%\begin{lemma}
% Given integers $f, g, t$ with $g,t\ge 1$ and $f\le 2g-1$,
%consider a Frobenius symbol
%\begin{equation*} 
%\l=\left( \begin{array}{ccc:ccc} a_1 &  \cdots & a_{t-1} & a_t & \cdots & a_\d \\ 
%b_1 & \cdots &b_{t-1} & b_t & \cdots &b_\d\end{array} \right)=LR,
%\end{equation*}
%such that $a_y - b_y \le f$ for all $1\le y \le t-1$ and  $a_t - b_t \ge f+1$ (and $b_{\d}\ge g-f-1$). Then
%the array
%\begin{equation*} 
%\mu=
%\left( \begin{matrix} \mu_{11} &  \mu_{12} & \cdots & \mu_{1{\d}'}\\ 
%\mu_{21} &  \mu_{22} & \cdots & \mu_{2{\d}'}\end{matrix} \right)
%:=
%s_{g+1}(L)\,d_{f-2g+1}(c_{g-f-1}(R))
%\end{equation*}
%satisfies the followings.
%\begin{enumerate}
%\item $\mu$ is a Frobenius symbol. \label{lem2:wd}%, i.e., $\l \mapsto \mu$ is well defined.
%\item $\mu_{1y} - \mu_{2y} \le f-2g-2$ for all $1\le y \le t-1$ and $\mu_{1t} - \mu_{2t} \ge f-2g-1$. \label{lem2:rk}
%\item The correspondence from $\l$ to $\mu$ is reversable. \label{lem2:re}
%\item $|\l|-|\mu|=2g-f$. \label{lem2:wt}
%\item $a_{t-1}\ge g+1 $. \label{lem2:it}
%\end{enumerate}
%\end{lemma}
\begin{proof}
First note that since $rank(R)=\a_1-\b_1\le f-2g+1$, $d_{f-2g+1}(R)$ is well defined.  We now prove each of the five statements. 

(1) %First we prove \eqref{lem2:wd}.
If $L=\emptyset$, then $\mu=d_{f-2g+1}(R)$ is obviously
a Frobenius symbol since the Dyson map is defined on partitions.
Now assume $L\ne\emptyset$.
Then the last column of  $s_{g+1}(L)$ is  
$\begin{matrix} a_{t-1}-g-1\\ 
b_{t-1}+g+1 \end{matrix}$. 
\begin{itemize}
\item
If $d_{f-2g+1}(R)\ne \emptyset$, then its first column is
$\begin{matrix} \b_1+f-2g \\ \g \end{matrix}$, 
where $\g$ is $0$, $1$, or $\b_{2}+2$ (see Section~\ref{dyson}). 
Since $a_{t-1}>\beta_1-g+f+1$, we have
\begin{equation*}
a_{t-1}-g-1> \b_1+f-2g,
\end{equation*}
from which with $a_{t-1}-b_{t-1}\le f$, it follows that
\begin{equation*}
b_{t-1}+g+1 \ge a_{t-1}-f+g+1>\b_1+2 > \g.
\end{equation*}
Thus $\mu=s_{g+1}(L)\,d_{f-2g+1}(R)$ is a Frobenius symbol.
\item
If $d_{f-2g+1}(R)=\emptyset$,  
then $R=\left(\begin{matrix} 0\\ 2g-f-1 \end{matrix}\right)=\left(\begin{matrix} \a_1\\ \b_1 \end{matrix}\right)$. Since $a_{t-1}>\beta_1-g+f+1$, 
\begin{equation*}
a_{t-1}-g-1 \ge \b_1 - 2g +f +1 = 0.
\end{equation*}
Also, since $b_{t-1}> \a_1+g-f-1$ and $2g\ge f+1$, 
\begin{equation*}
b_{t-1}+g+1 \ge \a_{1}+2g-f +1 \ge 2.
\end{equation*}
Thus
$\mu=s_{g+1}(L)$ is a Frobenius symbol.
\end{itemize}

(2) For $1\le y \le t-1$, since $a_y-b_y\le f$, 
\begin{equation*}
\mu_{1y}-\mu_{2y}=(a_y-b_y)-2g-2\le f-2g-2.
\end{equation*}

Also, by Remark~\ref{remark1}, 
\begin{equation*}
\mu_{1t}-\mu_{2t}=rank(d_{f-2g+1}(R))\ge (f-2g+1)-2.
\end{equation*}
%Thus we have proved \eqref{lem2:rk}.

(3) If $rank(L)\le h$, then $a_1-b_1 \le h$. Thus
$$
rank(\mu)=\mu_{11}-\mu_{21}=(a_1-b_1)-2g-2\le h-2g-2.
$$

(4) Consider a Frobenius symbol $\mu=\left( \begin{matrix} \mu_{11} &  \mu_{12} & \cdots & \mu_{1{\d}'}\\ 
\mu_{21} &  \mu_{22} & \cdots & \mu_{2{\d}'}\end{matrix} \right)$
satisfying \eqref{lem2:rk}. 
Set 
$$
L'=\left( \begin{matrix} \mu_{11} & \cdots & \mu_{1(t-1)}\\ 
\mu_{21} & \cdots & \mu_{2(t-1)}\end{matrix} \right) 
\quad\mbox{and}\quad
R'=\left( \begin{matrix} \mu_{1t}  & \cdots & \mu_{1 {\d}'}\\ 
\mu_{2t} & \cdots & \mu_{2{\d}'}\end{matrix} \right). 
$$

Then $ s_{-g-1}(L')$ and $d_{f-2g+1}^{-1}(R')$ are the desired Frobenius symbols. 
So \eqref{lem2:re} holds. 

(5) Finally, since $|s_{g+1}(L)|=|L|$ and $|d_{f-2g+1}(R)|=|R|-(2g-f)$, \eqref{lem2:wt} holds true.
%Finally, \eqref{lem2:it} is derived in the proof of eqref{lem2:wd}.
\end{proof}

\subsection{Bijection $\psi_m$} \label{psi}
Assume that $m$ is a positive integer. We now construct a bijection $\psi_m$ between the $(k,i)$-singular overpartitions of $n$  with $m$ dotted blocks and the partitions of $n'$, where $n'=n-k\binom{m}{2}-i m$  if the last dotted block is negative, and $n'=n-k\binom{m+1}{2} +i m$  if the last dotted block is positive. This proves Theorem~\ref{thm3.1}.  By symmetry, it is sufficient to show the case that the last dotted block is negative.

Let $\lambda$ be a $(k,i)$-singular overpartition in which there are exactly $m$ dotted blocks and the last dotted block is negative .

First let $D_1$ be the union of the last dotted block and the blocks on the right of the last dotted block if any.  From the right to left, denote each of the unchosen dotted blocks by $D_v$ for $1<v \le m$. Let $D_{m+1}$ be the union of the blocks on the left of $D_{m}$ if any. % and $T$ be the blocks on the right side of $D_1$.
For example, consider a $(5,2)$-singular overpartition
\begin{equation*}
\l=\left( \begin{array}{cc|cc|ccc|c|cc}
31 & 28 & 27 &  22 & 18  &  9  & 8 &  6  & 1 & 0 
\\ 
30 & 28 & 25 &  20 & 19  &  11 & 8 &  4 & 2 & 0 
\end{array} \right),
\end{equation*}
with its sequence of dotted blocks  $E\dot{P}\dot{N}PN$. Then we have
$$
D_3=\left( \begin{matrix} 31  & 28\\ 30  & 28 \end{matrix}\right),~
D_2=\left( \begin{matrix} 27  & 22\\ 25  & 20 \end{matrix}\right),~
D_1=\left( \begin{matrix} 18  & 9 & 8 & 6 & 1& 0 \\ 19  & 11 &8  & 4& 2& 0
\end{matrix}\right).
$$

We then define $\G_{1},\ldots,\G_{m+1}$ and $\psi_{m}(\l)$ as follows.
\begin{itemize}
\item Set $\G_{1}=D_1$.
\item  For $1\le v \le m$, set
\begin{equation*}
\G_{v+1}=\begin{cases}
c_{i+w k}(D_{v+1})\,d_{1-i-(v-1) k}(\G_{v}) 
& \mbox{if $v=2w+1$ for some $w\ge 0$,}\\
s_{w k}(D_{v+1})\,d_{1-i-(v-1)k}(\G_{v}) 
& \mbox{if $v=2w$ for some $w>0$.}\\
\end{cases}
\end{equation*}
%\item Set
%$\tilde\G_{m}=\begin{cases}
%%\G_{m} &\mbox{$\ell=m$,}\\
%c_{\frac{m+1}{2}i+\frac{m-1}{2} j}(H)\,\G_{m}
%& \mbox{$m$: odd,}\\
%s_{\frac{m}{2}i+\frac{m}{2} j}(H)\,\G_{m} 
%& \mbox{$m$: even,}
%\end{cases}$
\item Define $\psi_m(\l)=\G_{m+1}$.
\end{itemize}

Now we will inductively show that, for each $1 \le v \le m$, $\G_{v}$ is a partition 
satisfying 
\begin{equation} \label{condition}
rank(\G_{v}) \le 1-i-(v-1)k.
%\quad\mbox{and}\quad 
%|D_{v+1}\G_{v}|-|\G_{v+1}|=vi+(v-1)j.
\end{equation}
First, since $\G_1=D_1$ and the first column of $D_1$ is $(k,i)$-negative,  $\G_1$ is a partition satisfying \eqref{condition}. 
Assume that for $1\le v<m$,  $\G_{1},\G_{2},\ldots,\G_{v}$ are well defined and satisfy \eqref{condition}.
Consider $\G_{v+1}$. 
%Since $r(\G_{v})\le 1-vi-(v-1)j$, the Frobenius symbol $d_{1-vi-(v-1)j}(\G_{v})$ satisfies $r(d_{1-vi-(v-1)j}(\G_{v}))\ge -vi-(v-1)j-1$ and $|\G_{v}|-|d_{1-vi-(v-1)j}(\G_{v})|=vi+(v-1)j$.
\begin{itemize}
\item[Case 1:] Suppose $v=2w+1$ for some $w\ge0$. %Set $w=\frac{v-1}{2}$. 
Then we can write $\G_{v+1}=\G_{2w+2}$ as
$$
\G_{2w+2}=c_{i+wk}(D_{2w+2})\,d_{1-i-2wk}(\G_{2w+1}).
$$
In Lemma~\ref{lem:cd}, set $f=2-i$, $g=wk+1$, $h=k-i-1$, and $L=D_{2w+2}$, $R=\G_{2w+1}$.
Clearly $f,g,h$ satisfy $g\ge 1$, $2g\ge f+1$, $h\ge f$. 

Let us check the four conditions of the lemma.  First, note that $D_{2w+2}$ is not $(k,i)$-negative because the last dotted block is negative and the parity is alternating, so Condition i) holds true.  From the assumption \eqref{condition}, 
$$rank(\G_{2w+1})\le 1-i-2wk,$$ 
so Condition ii) holds true.  
Lastly,  let $\begin{matrix} x_1 \\ x_2  \end{matrix}$ and $\begin{matrix} z_1\\ z_2\end{matrix}$  be the last column of $D_{2w+2}$ and the first column of $D_{2w+1}$, respectively.  Since $D_{2w+2}D_{2w+1}$ forms a Frobenius symbol, we know
\begin{equation*}
x_1>z_1, \quad x_2>z_2. 
\end{equation*}
Also, since $\Gamma_{2w+1}=s_{wk}(D_{2w+1})d_{1-i -(2w-1)k}(\Gamma_{2w})$, the first column of $\Gamma_{2w+1}$ is $\begin{matrix} z_1-wk\\ z_2+wk \end{matrix}$. 
Thus, we have
\begin{align*}
x_1&>(z_1-wk)+ (wk+1)-1,\\
x_2&>(z_2+wk)- (wk+1)+1\ge 0,
\end{align*}
which verify Conditions iii) and iv).  Since all the four conditions in Lemma~\ref{lem:cd} are satisfied, by Statement \eqref{lem1:rank} of Lemma~\ref{lem:cd}, $\G_{2w+2}$ is a Frobenius symbol satisfying \eqref{condition}.

\item[Case 2:] Suppose $v=2w$ for some $w\ge1$. %Set $w=\frac{v-1}{2}$. 
Then we can write $\G_{v+1}=\G_{2w+1}$ as
$$
\G_{2w+1}=s_{wk}(D_{2w+1})\,d_{1-i-(2w-1)k}(\G_{2w}).
$$
In Lemma~\ref{lem:sd}, set $f=k-i-2$, $g=wk-1$, $h=1-i$, and $L=D_{2w+1}$, $R=\G_{2w}$.
Clearly, $f,g$ and $h$  satisfy $g\ge 1$, $2g\ge f+1$, $h\le f$.  Note that $D_{2w+1}$ is not $(k,i)$-positive, 
$$
rank(\G_{2w})\le 1-i-(2w-1)k,
$$ 
and $D_{2w+1}D_{2w}$ forms a Frobenius symbol. Thus, in the same way as Case 1, we can see all the four conditions in Lemma~\ref{lem:sd} are satisfied.
Therefore, by Statement \eqref{lem2:rank} of Lemma~\ref{lem:sd},  $\G_{2w+1}$ is a Frobenius symbol satisfying \eqref{condition}.
\end{itemize}

We now have that $\G_m$ is a Frobenius symbol satisfying \eqref{condition} from the induction. Also, $D_{m+1}$ is a Frobenius symbol.  We can easily check that $D_{m+1}$ and $\G_m$ satisfy the conditions for Lemmas~\ref{lem:cd} or \ref{lem:sd}.  Thus,  $\G_{m+1}$ is a Frobenius symbol by Statement~\eqref{lem1:wd} of each lemma.

%We now show that $\G_{m+1}$ is indeed a Frobenius symbol. 

Let us then check the weight difference. By Statement \eqref{lem2:wt} of each of Lemmas~\ref{lem:cd} and \ref{lem:sd}, we have 
\begin{equation} \label{condition1}
|D_{v+1}\G_{v}|-|\G_{v+1}|=i+(v-1)k
\end{equation}
for $v=1,\ldots, m$.
%Also, we can easily check that $\G_{m+1}$ becomes a Frobenius symbol satisfying 
%\begin{equation*} \label{Gv:condi}
%|D_{m+1}\G_{m}|-|\G_{m+1}|=mi+(m-1)j.
%\end{equation*}
By \eqref{condition1}, we have
\begin{align*}
|\l|
&=|D_{m+1} D_m \cdots D_5D_{4} D_{3} D_{2} D_{1} | \\
&=|D_{m+1} D_m \cdots D_5D_{4} D_{3} D_{2} \G_1| \\
&=|D_{m+1} D_m \cdots D_5D_{4} D_{3} \G_2|+i \\
&=|D_{m+1} D_m \cdots D_5D_{4} \G_3|+(i+k)+i \\
&=|D_{m+1} D_m \cdots D_5 \G_4 |+ (i+2k)+ (i+k)+i \\
& \;\; \vdots \\
&=|{\G}_{m+1}|+\sum_{v=1}^{m} \big( i +k  (v-1) \big).
\end{align*}
Thus 
\begin{equation*}
|{\G}_{m+1}| = n-k\binom{m}{2}- i m.
\end{equation*}

By Statement \eqref{lem1:re} of Lemma \ref{lem:cd}  and Lemma \ref{lem:sd},
each process of producing $\G_{v+1}$ is reversible. 
Therefore, $\psi_m$ is indeed a bijection. The inverse map will be given after the following example. 

\begin{example} \label{ex1}
Consider a $(5,2)$-singular overpartition
\begin{equation*}
\l=\left( \begin{array}{cc|ccc|ccccccc|cc|cc}
31 & 28 & 27 & 25 & 22 & 18 & 16 & 14 & 13 & 9  & 8 & 7 & 6 & 4 & 1 & 0 
\\ 
30 & 28 & 25 & 24 & 20 & 19 & 16 & 15 & 12 & 11 & 8 & 7 & 4 & 3 & 2 & 0 
\end{array} \right),
\end{equation*}
with its sequence of dotted blocks $EP\dot{N}\dot{P}\dot{N}$. Note that $k=5$, $i=2$, 
and $m=3$.  We have the following $\G_v$ for $v=1,2,3,4$:
\begin{itemize}
\item $\G_1 = D_1 = 
\left( \begin{array}{cc}
 1 & 0 
\\ 
 2 & 0 
\end{array} \right)$,
\item $\G_2 = c_2(D_2)\,d_{-1}(\G_1) =
\left( \begin{array}{cc:c}
 2 & 1 & 0 
\\ 
 8 & 6 & 2 
\end{array} \right)$,
\item $\G_3 = s_5(D_3)\,d_{-6}(\G_2) =
\left( \begin{array}{ccccccc:cc}
13 & 11 &  9 &  8 &  4 &  3 &  2 & 1 & 0
\\ 
24 & 21 & 20 & 17 & 15 & 13 & 12 & 8 & 4
\end{array} \right)$,
%\item $\G_4 = c_7(B_1)\,d_{-11}(\G_2) =
%\left( \begin{array}{ccc|cccccccc}
%18 & 17 &  13 &  12 &  11 &  9 &  7 & 6 & 2 & 1 & 0
%\\ 
%34 & 32 & 30 & 23 & 22 & 19 & 17 & 15 & 14 & 10 & 6
%\end{array} \right)$,
\item ${\G}_4 = c_7(D_4)\,d_{-11}(\G_3) =
\left( \begin{array}{ccccc:cccccccc}
23 & 21 & 18 & 17 &  13 &  12 &  11 &  9 &  7 & 6 & 2 & 1 & 0
\\ 
38 & 35 & 34 & 32 & 30 & 23 & 22 & 19 & 17 & 15 & 14 & 10 & 6
\end{array} \right)$,
\end{itemize}
where the dotted lines are put to separate the concatenated two arrays in each $\G_v$. 
Here ${\G}_4$ is the ordinary partition %$\mu$ 
corresponding to the $(5,2)$-singular overpartition $\l$.
\end{example}

We now briefly describe the inverse of $\psi_m$. 
Let $\mu$ be an ordinary partition. %  with prescribed $k, i$ and $m$. 
%we can recover the corresponding $(k,i)$-singular partition $\l$ as follows. 
\begin{itemize}
\item Set $\mu=\G_{m+1}$. 
\item  For $v=m,\ldots, 1$, let  $\begin{matrix} a_t \\ b_t \end{matrix}$ be the first column of $\G_{v+1}$ such that $a_t-b_t \ge -(v-1)k-i-1$.  If there exists no such $t$, then we define $t$ to be $1+ \ell(\G_{v+1})$, where $\ell(\G_{v+1})$ denotes the number of columns of $\G_{v+1}$. 
Split $\G_{v+1}$ into two arrays $L_v$ and $R_v$ by choosing the first $t-1$ columns for $L_v$ and the rest for $R_v$. % first $t-1$ columns of $ by adding vertical dash at just before the column.
Set 
$$D_{v+1}=\begin{cases}c_{i+wk}(L_v) & \mbox{if $v=2w+1$ for some $w\ge 0$,}\\
s_{-wk}(L_v) & \mbox{if $v=2w$ for some $w>0$,}\end{cases}$$ and 
$\G_v=d_{1-i-(v-1)k}^{-1}(R_v)$. 

\item Define $\psi_m^{-1}(\mu)=D_{m+1}\cdots D_2D_1.$
\end{itemize}
%Then we obtain $D_{m+1},\ldots, D_2$ and $\G_1=D_1$, and  $\l=D_{m+1}\cdots D_2 D_1$.

In the following example, we present how $\psi^{-1}_m$ works with the partition obtained in Example~\ref{ex1}. 
\begin{example} 
Let 
$$\mu =
\left( \begin{array}{ccccccccccccc}
23 & 21 & 18 & 17 &  13 &  12 &  11 &  9 &  7 & 6 & 2 & 1 & 0
\\ 
38 & 35 & 34 & 32 & 30 & 23 & 22 & 19 & 17 & 15 & 14 & 10 & 6
\end{array} \right).$$
Note that  we know that $k=5$, $i=2$, and $m=3$. %Then, 
Below a dotted line is used to separate the two arrays in each step. Also,  to the right of $\G_v$,  $a_t-b_t$ is given to indicate $t$, and to the right of $D_{v+1}\G_{v}$, two maps that are applied to get $D_{v+1}$ and $\G_v$ are given. 
\begin{itemize}
%\item $
%\left( \begin{array}{cc|ccccccccccc}
%23 & 21 & 18 & 17 &  13 &  12 &  11 &  9 &  7 & 6 & 2 & 1 & 0
%\\ 
%38 & 35 & 34 & 32 & 30 & 23 & 22 & 19 & 17 & 15 & 14 & 10 & 6
%\end{array} \right)$,
%$a_t-b_t \le -16$
%\item $
%\left( \begin{array}{cc|ccccccccccc}
%31 & 28 & 18 & 17 &  13 &  12 &  11 &  9 &  7 & 6 & 2 & 1 & 0
%\\ 
%30 & 28 & 34 & 32 & 30 & 23 & 22 & 19 & 17 & 15 & 14 & 10 & 6
%\end{array} \right)$,
%$c_7^{-1}$
\item 
$
\G_4=L_4R_4=
\left( \begin{array}{ccccc:cccccccc}
23 & 21 & 18 & 17 &  13 &  12 &  11 &  9 &  7 & 6 & 2 & 1 & 0
\\ 
38 & 35 & 34 & 32 & 30 & 23 & 22 & 19 & 17 & 15 & 14 & 10 & 6
\end{array} \right)$, $a_6-b_6 \ge -13$, 

%\noindent Since $-(m-1)k-i-1=-13$, we see that $t=6$, i.e.,  $12-23=-11\ge -13$, so 

%\noindent $
%\G_4=
%\left( \begin{array}{ccccc:cccccccc}
%23 & 21 & 18 & 17 &  13 &  12 &  11 &  9 &  7 & 6 & 2 & 1 & 0
%\\ 
%38 & 35 & 34 & 32 & 30 & 23 & 22 & 19 & 17 & 15 & 14 & 10 & 6
%\end{array} \right)$, $a_6-b_6 \ge -13$

\item $
D_4 \G_3=
\left( \begin{array}{ccccc:ccccccccc}
31 & 28 & 27 & 25 &  23 &  13 &  11 &  9 &  8 & 4 & 3 & 2 & 1 & 0
\\ 
30 & 28 & 25 & 24 & 20 & 24 & 21 & 20 & 17 & 15 & 13 & 12 & 8 & 4
\end{array} \right)$,\quad$c_{7}, \,  d_{-11}^{-1} $
\item $
\G_3=L_3R_3=
\left( \begin{array}{ccccccc:cc}
 13 &  11 &  9 &  8 & 4 & 3 & 2 & 1 & 0
\\ 
 24 & 21 & 20 & 17 & 15 & 13 & 12 & 8 & 4
\end{array} \right)$,\quad$a_8-b_8 \ge -8$,
\item $
D_3 \G_2=
\left( \begin{array}{ccccccc:ccc}
 18 & 16 & 14 & 13 & 9 & 8 & 7 & 2 & 1 & 0
\\ 
 19 & 16 & 15 & 12 & 10 & 8 & 7 & 8 & 6 & 2
\end{array} \right)$,\quad$s_{-5}, \, d_{-7}^{-1}$,
\item $
\G_2=L_2R_2=
\left( \begin{array}{cc:c}
 2 & 1 & 0
\\ 
 8 & 6 & 2
\end{array} \right)$,\quad$a_3-b_3 \ge -3$,
\item $
D_2\G_1=
\left( \begin{array}{cc:cc}
 6 & 4 & 1 & 0
\\ 
 4 & 3 & 2 & 0
\end{array} \right)$,\quad$c_{2}, \,d_{-1}^{-1}$
\item $
D_1 =\G_1=
\left( \begin{array}{cc}
 1 & 0
\\ 
 2 & 0
\end{array} \right)$.

\end{itemize}
Thus we recover the singular overpartition $\lambda$ in Example~\ref{ex1}:
$$
\l=
\left( \begin{array}{cccccccccccccccc}
31 & 28 & 27 & 25 & 23 & 18 & 16 & 14 & 13 & 9 & 8 & 7 & 6 & 4 & 1 & 0
\\ 
30 & 28 & 25 & 24 & 20 & 19 & 16 & 15 & 12 & 10 & 8 & 7 & 4 & 3 & 2 & 0
\end{array} \right),
$$
where
\begin{gather*}
D_4=\left( \begin{matrix}31 & 28 & 27 & 25 & 23\\ 
30 & 28 & 25 & 24 & 20 \end{matrix}\right),~
D_3=\left( \begin{matrix} 18 & 16 & 14 & 13 & 9 & 8 & 7\\ 
19 & 16 & 15 & 12 & 10 & 8 & 7 \end{matrix}\right),
D_2=\left( \begin{matrix} 6  & 4 \\ 4  & 3 \end{matrix}\right),~
D_1=\left( \begin{matrix} 1  & 0 \\ 2  & 0 \end{matrix}\right).
\end{gather*}
\end{example}

Let us give another (simple but non-trivial) example.
\begin{example}
Consider a partition 
$$\mu =
\left( \begin{array}{c}
0
\\ 
0
\end{array} \right).$$
Note that we have information $m=2$, but arbitrary $k$ and $i$. Then,
\begin{itemize}
%\item $
%\left( \begin{array}{|c}
%0
%\\ 
%0
%\end{array} \right)$,
%$a_t-b_t \le -3i-2j+1$. ($B_0=\emptyset$)
%\item $
%\left( \begin{array}{c|c}
%&0
%\\ 
%&0
%\end{array} \right)$,
%$s_{i+j}^{-1}$
\item $\G_3=L_3R_3=
\left( \begin{array}{c:c}
& 0
\\ 
& 0
\end{array} \right)$,\quad$a_1-b_1=0 \ge -i-k-1$, so $t=1$,

\item $D_3\G_2=
\left( \begin{array}{c: c}
& 0
\\ 
& i+k
\end{array} \right)$,\quad$s_{-k}, \,d_{1-i-k}^{-1},$

\item $\G_2=L_2R_2=
\left( \begin{array}{c:c}
0 &
\\ 
i+k &
\end{array} \right)$,\quad$a_1-b_1=-i-k <-i-1$, so $t=2$,

\item $
D_2 \G_1=
\left( \begin{array}{c:c}
 k&0
\\ 
 i &i-1
\end{array} \right)$,\quad$c_{i},\,  d_{1-i}^{-1}$,
\item $
D_1= \G_1=
\left( \begin{array}{c}
 0
\\ 
 i-1
\end{array} \right)$.
\end{itemize}
Thus we can restore 
$$
\l=
\left( \begin{matrix}%{cc}
 k&0
\\ 
i &i-1
\end{matrix} \right),
$$
where
\begin{gather*}
D_3=\emptyset,~D_2=\left( \begin{matrix} k \\ i \end{matrix}\right),~
D_1=\left( \begin{matrix}  0 \\ i-1 \end{matrix}\right).
\end{gather*} 
\end{example}

\section{The question of Andrews}\label{sec4}

Let $\overline{C}_{k,i}(n)$ be the number of overpartitions of $n$ in which no parts are multiples of $k$ and only parts congruent to $\pm i $ mod $k$ can be overlined. 
Theorem~\ref{thm1} says
\begin{equation*}
\overline{Q}_{k,i}(n)=\overline{C}_{k,i}(n). 
\end{equation*}

For any integer $m$, we let $\overline{C}_{k,i}(n,m)$ be the number of overpartitions counted by $\overline{C}_{k,i}(n)$ such that the number of overlined parts congruent to $ i$ mod $k$ minus the number of overlined parts congruent to $-i$ mod $k$  equals $m$. 

Also, for $m\ge 0$, let $\overline{Q}_{k,i}(n,m)$ be the number of $(k,i)$-singular overpartitions of $n$ with exactly $m$ dotted blocks and the last dotted block negative. For $m<0$,   let $\overline{Q}_{k,i}(n,m)$ be the number of $(k,i)$-singular overpartitions of $n$ with exactly $|m|$ dotted blocks and the last dotted block positive.  

We will prove the following refined version of Theorem~\ref{thm1}:
\begin{theorem} \label{thm4.1}
For any integer $m$, 
\begin{equation*}
\overline{Q}_{k,i}(n,m)=\overline{C}_{k,i}(n,m).
\end{equation*}
\end{theorem}

\proof
Let $\pi$ be an overpartition counted by $\overline{C}_{k,i}(n)$.  We first divide the parts of $\pi$ into three partitions $\mu^1, \mu^2, \gamma$ as follows: $\mu^1$ is the partition consisting of the overlined parts of $\pi$ that are congruent to $i$ mod $k$, $\mu^2$ is the partition consisting of the overlined parts of $\pi$ that are congruent to $-i$ mod $k$, and $\gamma$ is the partition consisting of the nonoverlined parts of $\pi$. Clearly, $\pi=\mu^1\cup \mu^2 \cup \gamma$ and $\ell(\mu^1)-\ell(\mu^2)=m$.

Recall the Wright map $\varphi$ from Section~\ref{wright}. Let $(\kappa,m)=\varphi(\mu^1, \mu^2)$. Then $\kappa$ is a partition into multiples of $k$. Clearly, $ \kappa \cup \gamma$ is  a partition of $n-k\binom{m}{2}-i m$. Thus we have
\begin{equation*}
\overline{C}_{k,i}(n,m)=p\bigg(n-k\binom{m}{2}-i m\bigg),
\end{equation*}
which with Theorem~\ref{thm3.1} completes the proof. 
\endproof

%Clearly, its generating function is:
%\begin{equation*}
%\sum_{n=0}^{\infty} \overline{C}_{k,i}(n) q^n = \frac{(-q^{i},-q^{k-i};q^{k})_{\infty}}{(q, q^2,\ldots, q^{k-1}; q^{k} )_{\infty}}.
%=\frac{(-q^{i},-q^{k-i},q^k;q^{k})_{\infty}}{(q; q )_{\infty}}.
%\end{equation*}

%%%%%%%%%%%%%
%\iffalse
%We will add $|d|(|d|-1)/2$ boxes to the $R$-modular Ferrers diagram of $\pi$. Suppose $d\ge 0$. We first put $d$ boxes on top of the Ferrers diagram of $\pi$, and then successively put $d-i$ boxes on the top of the $d-i+1$ added boxes for $i=1,\ldots, d-1$.   Along the main diagonal from the upper left box, we read off the number of boxes  on and below the diagonal in each column to form a partition $\mu^1$, and read off the number of boxes  to the right of the diagonal in each row to form a partition $\mu^2$. Let us suppose $d<0$. We first put $-d-1$ boxes vertically to the left of the Ferrers diagram of $\pi$, and then successively put $-d+i-1$ boxes vertically to the left of the $-d+i$ added boxes for $i=-1,\ldots, d+1$.   Along the main diagonal from the upper left box, we read off the number of boxes  below the diagonal in each column to form a partition $\mu^1$, and read off the number of boxes on and to the right of  the diagonal in each row to form a partition $\mu^2$.  If there is a box right below the last box on the diagonal but no box to the right of the diagonal, we add the zero part to $\mu^2$.  
%\fi
%%%%%%%%%%%%%%
%

\begin{remark}
Since only $\varphi$ and $\psi_m$ are used, the proof above is bijective, which answers to the question of Andrews.  
\end{remark}

Lastly, we illustrate how to combine $\varphi$ and $\psi_m$ to relate a $(k,i)$-singular overpartition of $n$ to an overpartition counted by $\overline{C}_{k,i}(n)$ in  an example.
\begin{example}
Let $\lambda$ be the  $(5,2)$-singular overpartition of  $469$ given in Example~\ref{ex1}, i.e, 
\begin{equation*}
\l=\left( \begin{array}{cc|ccc|ccccccc|cc|cc}
31 & 28 & 27 & 25 & 22 & 18 & 16 & 14 & 13 & 9  & 8 & 7 & {6} & 4 & 1 & 0 
\\ 
30 & 28 & 25 & 24 & 20 & 19 & 16 & 15 & 12 & 11 & 8 & 7 & {4} 
& 3 & {2} & 0 
\end{array} \right),
\end{equation*}  
with its sequence of dotted blocks is $E{P}\dot{N}\dot{P}\dot{N}$. Then we saw that  $m=3$ and 
\begin{align*}
\psi_3(\l) =\mu &=
\left( \begin{array}{ccccccccccccc}
23 & 21 & 18 & 17 &  13 &  12 &  11 &  9 &  7 & 6 & 2 & 1 & 0
\\ 
38 & 35 & 34 & 32 & 30 & 23 & 22 & 19 & 17 & 15 & 14 & 10 & 6
\end{array} \right)
\\
&=(24, 23, 21^2, 18^3, 17, 16^2, 13^9, 12^3, 11^3, 9, 8, 7^2, 5^6, 4, 3, 1^2), 
\end{align*}
where the power of a number indicates the number of occurrences of the number as a part. 
We now take the multiples of $5$ in $\mu$ to form the partition $(5,5,5,5,5,5)$. Then
\begin{equation*}
\varphi^{-1}((5,5,5,5,5,5),3)=((17, 12, 7, 2), ( 13 ))
\end{equation*}
Thus, we obtain the following overpartition $\pi$ counted by $\overline{C}_{5,2}(469)$:
\begin{align*}
\pi= (24, 23, 21^2, 18^3, \overline{17}, 17, 16^2, 13^9, \overline{13}, \overline{12}, 12^3, 11^3, 9, 8, \overline{7}, 7^2, 4, 3, \overline{2}, 1^2). 
\end{align*}
\end{example}

%%\section{Further Study}
%
%%Questions: Asymptotes of $a(n)$. Generalization of JTP. Truncation of other identities? 
%

\end{document}